\documentclass[oupthm]{CUP-JNL-CJM}
\usepackage{amsmath, amsfonts, amssymb, amscd, graphicx, amsthm, cite}
\usepackage{enumerate}
\usepackage{mathrsfs,url}
\usepackage{hyperref}
\usepackage{color}

\theoremstyle{plain}

    \newtheorem{thm}{Theorem}
    \newtheorem{theorem}[thm]{Theorem}

    \newtheorem{lemma}[thm]{Lemma}

    \newtheorem{corollary}[thm]{Corollary}

\theoremstyle{definition}

    \newtheorem{question}[thm]{Question}

    \newtheorem{definition}[thm]{Definition}

    \newtheorem{example}[thm]{Example}

\newcommand{\N}{{\mathbb N}}
\newcommand{\Z}{{\mathbb Z}}

\DeclareMathOperator{\id}{id}

\DeclareMathOperator{\Pol}{Pol}
\DeclareMathOperator{\Sym}{Sym}
 
\DeclareMathOperator{\rad}{rad}

\DeclareMathOperator{\CSP}{CSP}

\DeclareMathOperator{\arity}{ar}

\DeclareMathOperator{\Sg}{Sg}

\DeclareMathOperator{\Clo}{\mathsf{Clo}}
\newcommand{\alg}[1]{\mathbf{#1}}

\newcommand{\algA}{\alg{A}}
\newcommand{\algB}{\alg{B}}

\newcommand{\algG}{\alg{G}}

\DeclareMathOperator{\comP}{\mathsf{P}}

\DeclareMathOperator{\comEXPTIME}{\mathsf{EXPTIME}}

\DeclareMathOperator{\Decide}{\textsc{Decide}}

\theoremstyle{oupproof}
\newtheorem{proof}{Proof}

\newcommand\comment[1]{}

\begin{document}
\begin{Frontmatter} 
\author{Alexandr Kazda}
\author{Michael Kompatscher}
\authormark{A. Kazda and M. Kompatscher}

\address{\orgdiv{Department of Algebra, Faculty of Mathematics and Physics}, \orgname{Charles University}, \orgaddress{\street{Sokolovská 49/83}, \postcode{186 75} \city{Praha 8}, \country{Czech Republic}}\email{alex.kazda@gmail.com}\\ORCID 0000-0002-7338-037X}

\address{\orgdiv{Department of Algebra, Faculty of Mathematics and Physics}, \orgname{Charles University}, \orgaddress{\street{Sokolovská 49/83}, \postcode{186 75} \city{Praha 8}, \country{Czech Republic}}\email{michael@logic.at}}

\title[The local-global property for $G$-invariant terms]{The local-global property for $G$-invariant terms\thanks{This project was supported by grant No 18-20123S of the Czech Science Foundation (GA\v{C}R), the Charles University
   Research Center programs No.UNCE/SCI/022 and PRIMUS/21/SCI/014, and INTER-EXCELLENCE project LTAUSA19070 of the Czech Ministry of Education M\v{S}MT. Michael Kompatscher has further received funding from the European Research Council (ERC) under the European Union's Horizon 2020 research and innovation programme (grant agreement No 714532). The paper reflects only the authors' views and not the views of the ERC or the European Commission. The European Union is not liable for any use that may be made of the information contained therein.}}


\abstract{For some Maltsev conditions $\Sigma$ it is enough to check if a finite algebra $\algA$ satisfies $\Sigma$ \emph{locally} on subsets of bounded size in order to decide whether $\algA$ satisfies $\Sigma$ (\emph{globally}). This \emph{local-global property} is the main known source of tractability results for deciding Maltsev conditions. \\
In this paper we investigate the local-global property for the existence of a \emph{$G$-term}, i.e. an $n$-ary term that is invariant under permuting its variables according to a permutation group $G \leq \Sym(n)$. Our results imply in particular that all cyclic loop conditions (in the sense of Bodirsky, Starke, and Vucaj) have the local-global property (and thus can be decided in polynomial time), while symmetric terms of arity $n>2$ fail to have it.}
\end{Frontmatter}

\section{Introduction}
Maltsev conditions play a central role in universal algebra. In several classical results they were shown to characterise varieties of algebras with well-behaved congruence lattices \cite{HobbyMcKenzie}; more recently, they turned out to be an indispensable tool in the study of constraint satisfaction problems. By \cite{wonderland} the Maltsev conditions of height 1 that hold in the polymorphism algebra of a finite  structure $\mathbb A$ completely determine the complexity of $\CSP(\mathbb A)$.

Therefore it is a natural computational problem to decide whether a given algebra $\algA$ satisfies a fixed Maltsev condition $\Sigma$. The systematic study of the complexity of this problem was initiated by Freese and Valeriote in \cite{FV-decidingMC}. Unfortunately it is often a computationally hard problem, even for strong Maltsev conditions. It is for instance $\comEXPTIME$-complete to decide, whether a finite algebra has a semilattice operation, J\'onsson terms of fixed rank $n>2$, or a weak near unanimity operation of fixed arity \cite{FV-decidingMC, horowitz-decidingmaltsev}. The situation appears to be better if we restrict the input to finite \emph{idempotent} algebras; then there are several Maltsev conditions, which can be decided in polynomial time.

One of the main strategies in obtaining polynomial time algorithms is to check whether $\Sigma$ is satisfied locally in $\algA$ on subsets of bounded size. For several strong linear Maltsev conditions~$\Sigma$ this implies that $\algA$ satisfies $\Sigma$ on its whole universe. This \emph{local-global property} was used to prove the tractability of deciding whether a finite idempotent algebra has a Maltsev term \cite{FV-decidingMC}, a NU-term of fixed arity \cite{horowitz-decidingmaltsev}, J\'onsson terms and Gumm terms of fixed degree \cite{KV-pathconditions}, and $n$-Hagemann-Mitschke terms \cite{VW-decidingHM}. A recent result of the first author shows that 4-ary Siggers terms also have the local-global property. This is important because a finite algebra has a Siggers term if and only if it satisifes some non-trivial Maltsev condition of height~1  \cite{kazda-decidingQWNU}. In \cite{FV-decidingMC} and \cite{DFV-decidingdifferenceterms} it was shown that also some non-strong Maltsev conditions can be decided in polynomial time, by studying local conditions (i.e. properties of subalgebras with bounded generating sets). The only significant deviation from the local-global principle (in this broader sense) known to us is the proof that minority terms can be decided in NP \cite{KOVZ-minority} which uses a variation of Mayr's algorithm for the subpower membership problem in Maltsev algebras \cite{mayr-SMP}.

In this paper we study the local-global property for the height 1 Maltsev conditions $\Sigma_G$ that state the existence of a \emph{$G$-term}, i.e. a $n$-ary term $t(x_1,\ldots,x_n)$ that is invariant under permuting its variables according to the permutation group $G \leq \Sym(n)$. The idea to study Maltsev conditions parameterized by groups was suggested to us by Matt Valeriote. While for some groups $G$ the condition $\Sigma_G$ is quite familiar, we are not aware of any previous paper studying Maltsev conditions arising from the perspective of permutation groups. 

Such conditions $\Sigma_G$ are of interest because they connect universal algebra to the theory of permutation groups. Moreover, the conditions $\Sigma_G$ encompass many conditions relevant in the study of CSPs and PCSPs, such as the existence of cyclic terms, symmetric terms, and block-symmetric terms of fixed arity.

We show that $\Sigma_{G}$ has the local-global property whenever $G$ is a direct products of regular permutation groups. In particular this implies that we can decide in polynomial time whether a finite algebra has a \emph{cyclic term} of fixed arity, or, more generally, satisfies a fixed \emph{cyclic loop condition} (introduced in \cite{cyclic-loop-conditions}).

However, we also show that $\Sigma_{G}$ fails to have the local-global property (even for idempotent algebras) if $G$ has no fixpoint, but contains a permutation which fixes exactly one point, and has equisized orbits otherwise. This implies in particular that symmetric terms of arity $n > 2$ do not have the local-global property.  The only previously known strong Maltsev condition to not have the local-global property are minority terms \cite{KOVZ-minority}. 

We remark however that the failure of the local-global property does not imply anything about the complexity of deciding the existence of $n$-ary symmetric terms which still might be in $\comP$.

Additionally, we give a new application of the local-global property outside the realm of finite algebras: we show that whenever $\Sigma_G$ has the local-global property for finite algebras, then the local satisfaction of $\Sigma_G$ also lifts to global satisfaction in closed oligomorphic clones. Oligomorphic clones are clones on countably infinite sets that satisfy a certain compactness condition; they are essential to the study of CSPs of infinite structures (see \cite{Bodirsky-book} for background). 

Our paper is structured as follows: In Section \ref{sect:background} we formally define the problem of deciding Maltsev conditions and the local-global property. In Section \ref{sect:Gterms} we introduce $G$-terms, and make some observation on how they compare to each other within the interpretability lattice. In Section \ref{sect:example} we prove that $\Sigma_G$ has the local-global property for direct products of regular permutation groups while Section \ref{sect:failure} shows a failure of the local-global property for some other groups $G$. In Section \ref{sect:oligomorphic} we discuss the local-global property for oligomorphic clones. We conclude with Section \ref{sect:openproblems} which contains some open problems.

\section{Background} \label{sect:background}

\subsection{Preliminaries}
An algebra is a structure $\algA = (A,(f_i)_{i \in I})$ consisting of a non-empty set $A$, called the \emph{universe} of $\algA$, and a list of finitary operations $f_i \colon A^{k_i} \to A$, called the \emph{basic operations} of $\algA$.
If $\algA=(A,(f_i)_{i \in I})$ is an algebra, then $B\subset A$ is a \emph{subuniverse} of $B$ if $B$ is closed under $f_i$ for all $i\in I$. If $r_1,\dots,r_n\in A$ then the \emph{subuniverse of $\algA$ generated by $r_1,\dots,r_n$} is the smallest subuniverse of $\algA$ that contains $r_1,\dots,r_n$. We will denote this subuniverse by $\Sg_\algA(r_1,\dots,r_n)$. If $B$ is a nonempty subuniverse then the algebra $\algB$ that we get from $\algA$ by restricting all operations to $B$ is a \emph{subalgebra} of $\algA$.

An algebra is \emph{idempotent} if the identity $f_i(x,x,\ldots,x) = x$ holds for all basic operations of $\algA$ and all $x\in A$. In this paper we call an algebra \emph{finite}, if both the universe $A$ and the list of basic operations are finite. By $\arity(\algA)$ we denote the maximal arity of a basic operation of $\algA$.

An operation on $A$ is a \emph{term operation} of $\algA$ if it can be expressed as a composition of basic operations of $\algA$ and projections $p_i^n(x_1,\ldots,x_n) = x_i$. We write $\Clo(\algA)$ for the set (clone) of all term operations of $\algA$. Let $[n] = \{1,2,\ldots,n\}$. For any $n$-ary operation $t \colon A^n \to A$ and a map $\alpha \colon [n] \to [m]$ we define the \emph{minor of $t$ with respect to $\alpha$} as the operation $t^\alpha$ such that $t^{\alpha}(x_1,\ldots,x_m) = t(x_{\alpha(1)}, x_{\alpha(2)}, \ldots, x_{\alpha(n)})$ for all $x_1,\dots,x_m\in A$. It is an easy exercise to show that if $t\in \Clo(\algA)$, then $t^\alpha\in \Clo(\algA)$.

If $\alpha\colon [n] \to [m]$ and $\beta\colon [m]\to [k]$, then we have
$t^{\beta\circ\alpha}=(t^\alpha)^\beta$ (note the change of order). In particular
\[
t^{\beta\circ\alpha}(x_1,\dots,x_k)=(t^\alpha)(x_{\beta(1)},\dots,x_{\beta(m)}).
\]
If $\overline a$ is an $n$-tuple and $\sigma\colon[m]\to [n]$ a mapping, we denote by $\overline a^\sigma$ the $m$-tuple $(a_{\sigma(1)},a_{\sigma(2)},\dots,a_{\sigma(m)})$. With this notation, we get that $t^\sigma(\overline a)=t(\overline a^\sigma)$ for any $n$-ary operation $t$ and any $\overline a\in A^m$. In contrast to the composition order of minor-taking, we have $(\overline a^\alpha)^\beta=\overline a^{\alpha\circ\beta}$ for tuples. For future reference, note that for any $\overline a\in A^n$ and any permutations $\alpha,\beta\in S(n)$ we have the following set of identities
\[
t^{\alpha\beta}(\overline a)=
(t^\beta)^\alpha(\overline a)=t^\beta(\overline a^\alpha)=
t(\overline a^{\alpha\beta}).
\]

Whenever it is convenient (and does not risk confusion), we are alternatively going to label the variables of a term $t$ by finite index sets $I$ other than subsets of natural numbers, that is $t((x_i)_{i\in I})$.

We say an operation $f \colon A^n \to A$ \emph{preserves} a relation $R \subseteq A^m$, if $\overline r_1,\ldots,\overline r_n \in R$ implies $f(\overline r_1, \ldots, \overline r_n) \in R$; here the entries of $f(\overline r_1, \ldots , \overline r_k)$ are computed component-wise. The relation $R \subseteq A^m$ is \emph{invariant under $\algA$} if it is preserved by all operations of $\algA$.

When applying terms to tuples of tuples, it can be convenient to think of the tuples $\overline r_1, \ldots, \overline r_n$ as the columns of a matrix $M \in A^{m \times n}$. We will then use the notation $f(M) = f(\overline r_1, \ldots , \overline r_k)$; note that the $i$-th element of $f(M)$ is $f$ applied to the $i$-th row of $M$. It is not hard to see that for a fixed matrix $M = (\overline r_1, \ldots, \overline r_n) \in A^{m \times n}$ the set $\Sg_{\algA^m}(M) = \{t(M) \colon t \in \Clo(\algA) \} \subseteq A^m$ is the subuniverse generated by $\overline r_1, \ldots, \overline r_n$ in $\algA^m$. In Section~\ref{sect:example}, it will be convenient to index the rows or columns of a matrix $M\in A^{X\times Y}$ by other finite sets $X,Y$ than the natural numbers.

For a general background in universal algebra we refer to \cite{bergman-universal-algebra} or \cite{SB-universal-algebra}.

\subsection{Deciding Maltsev conditions}

An \emph{equation} (or identity) is a formal statement ``$t_1 \approx t_2$'' where $t_1$ and $t_2$ are terms constructed from some function and variable symbols. The $\approx$ symbol signifies that the equation should hold for all values of the variables. If $\algA$ is an algebra, $t_1$, $t_2$ are terms composed from the basic operations of $\algA$ and $x_1,x_2,\dots,x_n$ is the list of all variables occurring in $t_1$ and $t_2$, then we say that the equation $t_1\approx t_2$ \emph{holds} in $\algA$ if the sentence $\forall x_1\,\forall x_2\,\dots\forall x_n,\, t_1=t_2$ is satisfied in $\algA$.

For the purposes of this paper, a \emph{strong Maltsev condition} $\Sigma$ is a finite set of equations involving a finite set of variables $\{x_1,\ldots,x_n\}$ and a finite set of function symbols $\{d_1,\ldots,d_m\}$.  We are only going to study strong Maltsev conditions $\Sigma$ in this paper. (See \cite{bergman-universal-algebra} for the definition of general Maltsev conditions.)

An algebra \emph{$\algA$ satisfies $\Sigma$} if for every symbol $d_i$ in $\Sigma$ there is a term operation $d_i^{\algA} \in \Clo(\algA)$ of the same arity, such that each equation in $\Sigma$ holds in $\algA$ for the operations $d_1^{\algA},\ldots,d_m^{\algA}$. We then write $\algA \models \Sigma$ for short.

We say that \emph{$\algA$ satisfies $\Sigma$ on a set $F \subseteq A^n$}, if there are terms $d_1^\algA,\ldots,d_m^\algA \in \Clo(\algA)$, such that the sentence
$\forall (x_1,\ldots,x_n)\in F,\, t_1\approx t_2$ is satisfied in $\algA$ for each equation $t_1\approx t_2$ in $\Sigma$.

\begin{example}
$\Sigma= \{p(p(x_1, x_2), r(x_1)) \approx x_2 \}$ is a strong Maltsev condition involving the variables $x_1,x_2$, a binary operation symbol $p$ and a unary operation symbol $r$. Any Abelian group $\algG = (G, + , 0, -)$ satisfies $\Sigma$, as witnessed by the term operations $p^{\algG}(x_1,x_2) = x_1 + x_2$ and $r^{\algG}(x_1) = -x_1$.\\ The 2-element semilattice $\algA = (\{0,1\},\land)$ does not satisfy $\Sigma$, but it satisfies $\Sigma$ on the set $F= \{(0,0), (1,0)\}$, as witnessed by $p^{\algA}(x_1, x_2) = x_1 \land x_2$ and $r^{\algA}(x_1) = x_1$.
\end{example}

The interpretability quasiorder on Maltsev conditions is defined by $\Sigma_1 \leq \Sigma_2$, if for every algebra $\algA$ we have $\algA \models \Sigma_2 \Rightarrow \algA \models \Sigma_1$. We say that two Maltsev conditions $\Sigma_1, \Sigma_2$ are \emph{equivalent} if $\Sigma_1 \leq \Sigma_2$ and $\Sigma_2 \leq \Sigma_1$. Modulo this equivalence relation, the interpretability quasiorder forms a complete lattice.

A Maltsev condition $\Sigma$ is called \emph{trivial} if it is satisfied in every algebra (and thus minimal with respect to the interpretability order). A Maltsev condition is called of \emph{height 1} if its equations only involve minors of function symbols, and it is called \emph{linear} if its equations only involve minors of the functions symbols and variables.  For example, the Maltsev condition $p(y, x, x) \approx p(x,y,x) \approx p(x,x,y) \approx y$ is linear, but not of height 1. \\

For a Maltsev condition $\Sigma$ we define $\Decide(\Sigma)$ as the following decision problem:\\

\noindent \fbox{\parbox{\linewidth}{
$\Decide(\Sigma)$\\
Input: A finite algebra $\algA = (A, f_1,\ldots,f_n)$\\
Question: Does $\algA \models \Sigma$? }}\\

Here the input is given by the operation tables of $f_1,\ldots,f_n$. Thus its size can be measured by $\| \algA \|= \sum_{i = 1}^n |A|^{k_i}$, where $k_i$ is the arity of $f_i$. If we restrict the input to idempotent algebras we obtain the problem:\\

\noindent \fbox{\parbox{\linewidth}{ $\Decide^{id}(\Sigma)$\\
Input: A finite idempotent algebra $\algA = (A, f_1,\ldots,f_n)$\\
Question: Does $\algA \models \Sigma$? }}\\

We are only going to study strong Maltsev conditions $\Sigma$ in this paper. Then both $\Decide(\Sigma)$ and $\Decide^{id}(\Sigma)$ are decidable in $\comEXPTIME$; this follows directly from the fact that all operation in $\Clo(\algA)$ of arity bounded by some $m$ can be computed in time $C \arity(\algA) |A|^{|A|^m}$.

The idempotent problem $\Decide^{id}(\Sigma)$ trivially reduces to $\Decide(\Sigma)$, however in general the complexity of $\Decide(\Sigma)$ can be harder than $\Decide^{id}(\Sigma)$, even for linear Maltsev conditions \cite{FV-decidingMC, horowitz-decidingmaltsev}.

We would also like to point out that $\Decide(\Sigma)$ is different from the `Meta-Problem for CSPs', in which the input is given by a \emph{relational structure $\mathbb A$} instead of an algebra, and the task is to decide, whether the polymorphism clone $\Pol(\mathbb A)$ satisfies $\Sigma$ or not. A survey on the Meta-Problem for CSPs can be found in \cite{CL-metaproblem}.

\subsection{The local-global property}
We next formally define the local-global property and show how it can be used to prove the tractability of deciding linear Maltsev conditions.

\begin{definition} \label{definition:local-global} \
\begin{enumerate}
\item For $k\geq 1$, we say that a strong Maltsev condition $\Sigma$ with variable set $\{x_1,\dots,x_n\}$ has the \emph{$k$-local-global property} if any algebra $\algA$ such that $|A|^n\geq k$  and $\algA$ satisfies $\Sigma$ on all $k$ element sets $F \subseteq A^n$ satisfies $\Sigma$ on \emph{every} finite subset of $A^n$. In particular this implies $\algA \models \Sigma$ if $\algA$ is finite. 
\item We say that $\Sigma$ has the \emph{local-global property} if there is a $k$, such that $\Sigma$ has the $k$-local-global property.
\item We say that $\Sigma$ has the ($k$-)local-global property for a class of algebras $\mathcal K$ if the above is true for all algebras $\algA \in \mathcal K$.
\end{enumerate}
\end{definition}

We remark that our definition of the ($k$-)local-global property is consistent with several preceding definitions, such as the `local-global property of size $k$ for special cube terms', given in \cite{horowitz-decidingmaltsev}. In the special case of linear Maltsev conditions $\Sigma$, the local-global property gives rise to a polynomial time algorithm for $\Decide(\Sigma)$, by the following lemma:

\begin{lemma} \label{lemma:local-global}
Let $\Sigma$ be a strong linear Maltsev condition that has the $k$-local-global property. Let $n$ be the number of variables, and let $m$ be the number of different minors appearing in $\Sigma$. Then $\Decide(\Sigma)$ can be decided in time $\mathcal O (\arity(\algA) \|\algA\|^{k \cdot (m+n)})$.\\
The analogue statement holds for $\Decide^{id}(\Sigma)$ if $\Sigma$ has the $k$-local-global property for idempotent algebras.
\end{lemma}

\begin{proof}
If $\Sigma$ is empty, it will always be satisfied, so we can solve the problem in a constant time by just outputting ``Yes.'' This is why we will assume that $m,n$ are at least 1. In the rest of the proof we will assume that all the identities in $\Sigma$ are of the form $f^\sigma\approx g^\tau$ where $f$ and $g$ are operation symbols and the maps $\sigma$ and $\tau$ have $[n]$ as their codomain. Generally, the codomains might be smaller, but we can redefine them to be $[n]$ without affecting satisfaction or local satisfaction of $\Sigma$.

Let $\algA$ be an input to $\Decide(\Sigma)$; we want to check whether $\algA \models \Sigma$. Since $\Sigma$ has the local-global property of rank $k$, we only need to check whether $\algA$ satisfies $\Sigma$   on every set $F \subseteq A^n$ with $|F|=k$ (if $|A|^n<k$, we choose $F=A^n$ and use the procedure in the following paragraphs to look for terms satisfying $\Sigma$ on the whole $A^n$). The number of such sets $F$ is at most $|A|^{kn}$.
  
We want to decide if $\algA$ satisfies $\Sigma$ on a given $F$. For a fixed $\ell$-ary function symbol $f \in \Sigma$, let $M(f)$ be the set of all maps $\pi \colon   [\ell] \to [n]$, such that the minor $f^{\pi}$ appears in some identity in $\Sigma$. Consider now the matrix $S_f$ whose rows enumerate all tuples $(a_{\pi(1)},\ldots, a_{\pi(\ell)})$, such that $\overline a = (a_1,\ldots,a_n) \in   F$ and $\pi \in M(f)$. Since we have $k$ choices of $\overline a$ and $|M(f)|$ choices of $\pi$, the matrix $S_f$ has exactly $k \cdot |M(f)|$ rows (some rows might repeat). Denote the columns of $S_f$ by $(\overline s  _1,\ldots, \overline s_\ell )$.
  
  The subuniverse $\Sg(S_f)$ of $\algA^{k \cdot |M(f)|}$ generated by $\{\overline s_1,\ldots, \overline s_\ell\}$ has size at most $|A|^{k \cdot |M(f)|}$, and can be computed in time $C' \arity(\algA) \| \algA \|^{k \cdot   |M(f)|}$ (cf. Proposition 6.1. in \cite{FV-decidingMC}) for some constant $C' > 0$. Note that a tuple $\overline u$ is in $\Sg(S_f)$ if and only if there is a $t \in \Clo(\algA)$, such that the entries of $\overline u$ are of the form $t(a_{\pi(1)},\ldots, a_{\pi(\ell)})$ where $\overline a$ ranges over $F$ and $\pi$ over $M(f)$.
  
  Thus we can check whether $\Sigma$ is satisfied on $F$ by first computing $\Sg(S_{f_i})$ for every function symbol $f_i$ and then trying out all tuples $\overline s_1 \in S_{f_1}, \ldots, \overline s_n \in S_{f_n}$ to see if there is a choice of $s_1,\dots,s_n$ so that the equations $\Sigma$ hold with respect to their entries. This can be done in time
\[
C' \arity(\algA) \sum_{f \in \Sigma}  \|\algA \|^{k \cdot |M(f)|}+ C''\prod_{f \in \Sigma} |A|^{k \cdot |M(f)|} \leq C \arity(\algA) \|\algA \|^{km},
\]
  for some constant $C > 0$ (recall that $m=\sum_{f\in \Sigma}|M(f)|$). Running this test for every subset $F \subseteq A^n$ with $|F|=k$ gives us an algorithm that decides $\algA\models\Sigma$ in time $C \arity  (\algA) \| \algA \|^{km} \cdot |A|^{kn} \leq C \arity(\algA) \| \algA \|^{k(m+n)}$.
\end{proof}

We remark that in all known examples the proofs of the local-global property are constructive, i.e., they inductively construct terms that satisfy a Maltsev condition $\Sigma$ on bigger and bigger subsets of the domain of $\algA$.

For the existence of Maltsev terms this was explicitly pointed out in \cite{KOVZ-minority}: there it was even demonstrated that if $\algA$ has local Maltsev terms, then a circuit representation\footnote{Note that the representation matters here, the same might not be true for representation by terms.} of a global Maltsev term can computed in polynomial time. The same is also true for our local-global results in Section \ref{sect:example} with a proof similar to the one in~\cite{KOVZ-minority}; we omit the details to save space.

\section{$G$-invariant terms} \label{sect:Gterms} 

\begin{definition} \label{definition:Gterms}
Let $G \leq \Sym(n)$ be a permutation group on the set $[n] = \{1,2,\ldots,n\}$. We then say that an algebra $\algA$ has a \emph{$G$-term} $t \in \Clo(\algA)$ if for all 
$\pi \in G$:
\[
\algA \models t(x_1,\ldots,x_n) \approx t^{\pi}(x_1,\ldots,x_n).
\]
We will denote the corresponding Maltsev condition by $\Sigma_G$ for short. 
\end{definition}

Note that if $P$ is a set of generators of $G$, then already the identities $\algA \models t \approx t^{\pi}$ for all $\pi \in P$ imply that $t$ is a $G$-term (in particular, if $|P| = 1$, then $\Sigma_G$ is equivalent to a loop condition, and more general, if $|P|=m$, then $\Sigma_G$ is equivalent to a loop condition of width $m$, as defined in \cite{GJP-pseudoloop}).
Nevertheless, in this paper $\Sigma_G$ always denotes the entire set of identities $t \approx t^\pi$ where $\pi$ ranges over $G$.

Whenever it is convenient to us, we are going to extend Definition \ref{definition:Gterms} to groups $G$ that act on finite sets that are not of the form $\{1,2,\ldots,n\}$; should never cause confusion. 

\begin{example}
Several Maltsev conditions that were studied in the context of (promise) constraint satisfaction problems can be stated as $\Sigma_G$ for some $G$:
\begin{itemize}
\item $n$-ary \emph{cyclic terms} $c(x_1,x_2,\ldots,x_n) \approx c(x_2,x_3,\ldots,x_n,x_1)$ are $\Z_n$-terms, where $\Z_n \leq \Sym(n)$ is the group generated by the cyclic shift $(1,2,3,\dots,n)$.
\item \emph{cyclic loop conditions} (introduced in \cite{cyclic-loop-conditions}) are equivalent to the identities $\Sigma_{\langle \pi \rangle}$ for a single permutation $\pi \in \Sym(n)$;
\item A $n$-ary \emph{symmetric term} is a $\Sym(n)$-term;
\item A $n$-ary \emph{block-symmetric term} with respect to blocks $B_1, \ldots, B_k$ (see e.g. \cite{BGWZ-blocksymmetric}) is a $G$-term, where the direct product $G = \Sym(B_1) \times \cdots \times \Sym(B_k)$ acts naturally on $B_1 \dot\cup \cdots \dot\cup B_k = [n]$.
\end{itemize}
\end{example}

In this section we present some results on how the Maltsev conditions $\Sigma_G$ compare to each other in the interpretability order.\\

If not explicitly stated otherwise, we consider the direct product $G \times H$ of two permutation groups $G \leq \Sym(X), H \leq \Sym(Y)$ to be the permutation group $G \times H \leq \Sym(X \dot{\cup} Y)$, with the natural action 
\[ (\pi,\phi)(z) = \begin{cases} \pi(z) \text{ if } z \in X,\\
\phi(z) \text{ if } z \in Y,\\
\end{cases}
\]
for $\pi \in G, \phi \in H$.

Furthermore, recall that the \emph{wreath product} $G \wr H \leq \Sym(X \times Y)$ of two permutation groups consists of the elements of $G^{Y} \times H$, acting on $\Sym(X \times Y)$, by \[ ((\pi_y)_{y \in Y},\phi)(a,b) = (\pi_b(a),\phi(b)), \text{ for } (a,b) \in X \times Y.
\]

The following basic observations hold:

\begin{lemma} \label{lemma:G-terms}  \
\begin{enumerate}
\item $\Sigma_G$ is trivial if and only if $G$ fixes a point.
\item \label{itm:subgroup} If $G \leq H \leq \Sym(n)$, then every $H$-term is also a $G$-term. 
\item \label{itm:quotient} Let $G \leq \Sym(n)$ and $H\leq \Sym(m)$. If there are surjective maps $h \colon G \to H$ and $\alpha \colon [n] \to [m]$ such that $h(\pi) \circ \alpha = \alpha \circ \pi$ for all $\pi\in G$, then $\Sigma_{H} \leq \Sigma_G$.
\item \label{itm:conjugates}If $G,H\leq \Sym(n)$ are conjugate in $\Sym(n)$, then $\Sigma_G$ and $\Sigma_H$ are equivalent.
\item\label{itm:wreath} 
$\Sigma_{G \wr H}$ is the join of  $\Sigma_G$ and $\Sigma_H$ in the interpretability lattice.

\item \label{itm:products} 
$\Sigma_{G \times H}\leq \Sigma_G \land \Sigma_H$ where $\land$ is the meet in the interpretability lattice.
\end{enumerate}
\end{lemma}

\begin{proof} 
We will proceed point by point.
\begin{enumerate}
\item If $G$ has a fixpoint $i$, then the $i$-th projection $p_i^n(x_1,\ldots,x_n) = x_i$ is a $G$-term, and thus $\Sigma_G$ is trivial. If $G$ has no fixpoint, then it is not hard to see that there is no projection that could be a $G$-term. Thus the algebra on the universe $\{0,1\}$ whose all operations are projections cannot satisfy $\Sigma_G$. Thus $\Sigma_G$ is not trivial.
\item This holds trivially since $G \leq H$ implies $\Sigma_G \subseteq \Sigma_H$.
\item Let $\algA$ be an algebra that contains the $G$-term $f(x_1,\ldots,x_n)$. Let $g = f^\alpha$; we claim that $g$ is a $H$-term in $\algA$. Since $h$ is surjective, each element of $H$ can be written in the form $h(\pi)$ for some $\pi\in G$. To verify that $g$ is a $H$-term, we will thus show that $g^{h(\pi)}=g$ for any $\pi\in G$.

Recall that $h(\pi)\circ \alpha=\alpha\circ \pi$. Therefore, we get
\[
g^{h(\pi)} = \left(f^\alpha\right)^{h(\pi)} = f^{h(\pi) \circ \alpha} = f^{\alpha \circ \pi} = \left(f^{\pi}\right)^{\alpha} = f^{\alpha} = g.
\]
This shows that $g$ is indeed a $H$-term.
\item This follows directly from (3). Choose $\alpha$ so that $G=\alpha H \alpha^{-1}$ and let $h(\pi)=\alpha\circ \pi\circ \alpha^{-1}$.
\item Let $t_G(x_1,\ldots,x_n)$ be a $G$-term and $t_H(x_1,\ldots,x_m)$ be a $H$-term. Then $$t_H( t_G(x_{1,1}, x_{2,1}, \ldots, x_{n,1}) , \ldots, t_G(x_{1,m}, x_{2,m}, \ldots, x_{n,m}))$$
is a $(G \wr H)$-term.
On the other hand, if $s(x_{1,1}, x_{1,2}, \ldots, x_{n,m})$ is a $(G \wr H)$-term, then $s(x_1,\ldots,x_n,\ldots, x_1,\ldots,x_n)$ is a $G$-term, and $s(x_1,\ldots,x_1,\ldots, x_m,\ldots,x_m)$ is a $H$-term.

\item It is enough to show $\Sigma_{G\times H}\leq \Sigma_{G},\Sigma_{H}$. Let $t_G(x_1,\ldots,x_n)$ be a $G$-term. By adding $m$-many dummy variables, we obtain the $G \times H$-term $t_{G \times H}(x_1,\ldots,x_n,y_1,\ldots,y_m) := t_G(x_1,\ldots,x_n)$. This shows that $\Sigma_{G\times H}\leq \Sigma_G$. The proof that $\Sigma_{G\times H}\leq \Sigma_H$ is symmetrical.
\end{enumerate}
\end{proof}

Lemma \ref{lemma:G-terms} can be used to prove that for the prime decomposition $n = p_1^{i_1} \cdot p_2^{i_2} \cdots p_n^{i_n}$ the cyclic term $\Sigma_{\Z_n}$ is equivalent to the join of the conditions $\Sigma_{\Z_{p_1}}$, \ldots,  $\Sigma_{\Z_{p_n}}$. This was already known before, see for instance \cite{mirek-scloops}, where it is further shown that $\Sigma_{\Z_n} \leq \Sigma_{\Z_m}$ if and only if every prime divisor of $n$ is also a prime divisor of $m$.

More generally, the cyclic loop conditions, i.e., the identities $\Sigma_{\langle \pi \rangle}$, for groups with one generator $\pi$ were completely classified up to interpretability in \cite{cyclic-loop-conditions}:
\begin{theorem}[Theorem 5.10 and 5.23 in \cite{cyclic-loop-conditions}] \label{theorem:loop-condition}\
\begin{enumerate}
\item $\Sigma_{\langle \pi \rangle} \geq \Sigma_{\langle \rho \rangle}$ if and only if for every orbit of $\pi$ that has length $n$, $\rho$ has an orbit of length $m$ such that the radical $\rad(m)$ divides $\rad(n)$.
\item Moreover $\Sigma_{\langle \pi \rangle}$ is the join of all $\Sigma_{\langle \rho \rangle}$ such that $\Sigma_{\langle \pi \rangle} \geq \Sigma_{\langle \rho \rangle}$ and $\langle \rho \rangle$ has only orbits of distinct prime lengths.
\end{enumerate}
\end{theorem}

We next generalize some of these results to groups with more than one generator, starting with the following result about $p$-groups:

\begin{lemma} \label{lemma:p-Sylow}
Let $p$ be a prime. Then
the cyclic identities $\Sigma_{\Z_p}$ imply $\Sigma_G$ for every $p$-group $G$. If $G$ is a $p$-group that contains a permutation with no fixpoint, then $\Sigma_G$ is equivalent to $\Sigma_{\Z_p}$.
\end{lemma}

\begin{proof}
For every $n \in \N$, let $S_p(n)$ denote a Sylow $p$-subgroup of $\Sym(n)$. It is well known that $S_p(p)$ is equal to the cyclic group $\Z_p$, and $S_p(p^{k+1})$ is equal to the wreath product $S_p(p^{k}) \wr S_p(p)$ for every $k$. Repeated application of part (\ref{itm:wreath}) of Lemma \ref{lemma:G-terms} thus gives us that the condition $\Sigma_{S_p(p^k)}$ is equivalent to $\Sigma_{\Z_p}$ for every $k\in\N$.

For a general $n \in \N$, with base $p$ expansion $n = a_0 + a_1p + \cdots + a_{k}p^k$, we have $S_p(n) = \prod_{i = 0}^k (S_p(p^i))^{a_i}$ (by a result of Kalojnine~\cite{kalojnine-sylowSym}, see also~\cite[p. 176]{rotman-grouptheory}). By item (\ref{itm:products}) of Lemma \ref{lemma:G-terms}, the condition $\Sigma_{\Z_p}$ implies $\Sigma_{S_p(n)}$ for every $n\geq p$. Every $p$-group $G \leq \Sym(n)$ is (up to conjugation) a subgroup of $S_p(n)$, so by items (\ref{itm:subgroup}) and (\ref{itm:conjugates}) of Lemma \ref{lemma:G-terms}, $\Sigma_{\Z_p}$ implies $\Sigma_{G}$.

For the second part, it remains only to show that $\Sigma_G$ implies $\Sigma_{\Z_p}$. Assume that $G$ is a $p$-group and that there is a permutation $\pi \in G$ with no fixpoint.  Without loss of generality  assume that that the decomposition of $\pi$ into disjoint cycles has the form
\[
\pi=(1,2,\dots, a_1)(a_1+1, a_1+2, \dots, a_2) \cdots (a_m, a_m+1,\dots,n)
\]
for some suitable $a_1,\dots,a_m$.
Since $G$ is a $p$-group, the lengths of all orbits of $\pi$ need to be powers of $p$. Let $t$ be a $G$-term. Then the $\pi$-invariance of $t$ gives us that the minor of $t$
\[
c(x_1,\ldots,x_p) = t(x_1,x_2,\ldots,x_p,x_1,\dots,x_p,\ldots ,x_1,\ldots,x_p)
\] 
will be a $\Z_p$-term. Hence $\Sigma_G$ implies $\Sigma_{\Z_p}$ and we are done.
\end{proof}

An interesting class of conditions within the scope of Lemma~\ref{lemma:p-Sylow} are the \emph{'doubly cyclic'} identities given by the wreath-product $\Z_p \wr \Z_p$. These conditions recently found an application in the study of finitary tractable PCSPs~\cite{AB-finitelytractable}. We remark that Lemma \ref{lemma:p-Sylow} does not exclude the possibility of $p$-groups $G$, such that $\Sigma_G$ is non-trivial, but strictly weaker than the existence of $p$-cyclic terms. The smallest possible candidate is the group $G = \langle (12)(34), (12)(56) \rangle \leq \Sym(6)$ for $p = 2$.

Another example of $G$-term conditions equivalent to cyclic terms are the terms given by even dihedral groups:

\begin{lemma}\label{lem:D2n} Let $n\in\N$ and let $D_{2n} = \langle \pi, \phi \rangle \leq \Sym(2n)$ be the dihedral group of order $2n$,  where $\pi = (1,2, \dots, 2n)$ and $\phi = (1,2n)(2, 2n-1) \cdots (n, n+1)$ (so $D_{2n}$ can be regarded as the symmetry group of a $2n$-gon, acting on its vertices $\{1,2,\ldots,2n\}$). Then $\Sigma_{D_{2n}}$ is equivalent to $\Sigma_{\Z_{2n}} = \Sigma_{\Z_{2}} \lor \Sigma_{\Z_{n}}$.
\end{lemma}

\begin{proof}
Since $\Z_{2n}$ is a subgroup of $D_{2n}$, clearly $\Sigma_{D_{2n}}\geq \Sigma_{\Z_{2n}}$. For the other direction, let $c_{2n}$ be a $\Z_{2n}$-term. Then $c_2(x,y)=t(x,x,\dots,x,y,\dots,y)$ is clearly also cyclic of arity $2$. We define
\[t(x_1,\ldots,x_{2n}) = c_2(c_{2n}(x_1,x_2,\ldots,x_{2n}), c_{2n}(x_{2n}, x_{2n-1},\ldots,x_1))
\] It is not hard to see that $t$ is invariant under both $\pi$ and $\phi$.  Thus $t$ is a $D_{2n}$-term.
\end{proof}

Our later results about dihedral groups of odd degree $D_{2n+1}$ (Part (\ref{itm:dihedral}) of Corollary~\ref{corollary:failure}) indicates, however, that no analogical statement is true for them (see also Question \ref{question:equivalence}).

In the last lemma of this section we study $\Sigma_G$ for \emph{regular permutation groups $G \leq \Sym(G)$}, that is, groups that act on themselves by the left translation $\pi(\phi) = \pi \circ \phi$. 

\begin{lemma} \label{lemma:regular}
Let $G \leq \Sym(G)$ be a regular permutation group, $N \trianglelefteq G$ be a normal subgroup and $K = G/N$. Let us regard $N\leq \Sym(N)$ and $K\leq \Sym(K)$ both also as regular permutation groups. Then $\Sigma_G \geq \Sigma_N \lor \Sigma_K$. \\ If furthermore $G \cong N \times K$ (as abstract groups), then $\Sigma_G = \Sigma_N \lor \Sigma_K$.
\end{lemma}

\begin{proof}
Let $\kappa_1,\ldots,\kappa_q$ be a list of elements of $G$ such that $\{ N \kappa_1, N\kappa_2,\ldots, N \kappa_q \}$ enumerates the (right) cosets of $N$.  We group the variables of a $G$-term $f$ in blocks that correspond to the cosets of $N$,  i.e.,  $f((x_{\nu \kappa_1})_{\nu \in N}, \ldots, (x_{\nu \kappa_q})_{\nu \in N})$.

Since $f$ is a $G$-term, and $G$ acts by left translation, we get that for every $\mu \in N$: $f((x_{\mu \nu \kappa_1})_{\nu \in N}, \ldots, (x_{\mu \nu \kappa_q})_{\nu \in N}) = f((x_{\nu \kappa_1})_{\nu \in N}, \ldots, (x_{\nu \kappa_q})_{\nu \in N})$. This directly implies that the minor $f((x_\nu)_{\nu \in N}, \ldots, (x_\nu)_{\nu \in N})$ is an $N$-term.  Hence $\Sigma_G$ implies $\Sigma_N$.

In order to see that $\Sigma_G$ implies $\Sigma_K$,  we use criterion (\ref{itm:quotient}) of Lemma \ref{lemma:G-terms}. We define both $h \colon G \to K$ and $\alpha \colon G \to K$ to be the quotient map $\alpha(\pi) = h(\pi) = \pi N$. Then $h(\pi) \circ \alpha (\rho) = \pi N \circ \rho N = (\pi \circ \rho)N = \alpha \circ \pi (\rho)$, for all $\pi, \rho \in G$. Thus $\Sigma_G$ implies $\Sigma_K$. Together with the above, this shows that $\Sigma_G \geq \Sigma_N \lor \Sigma_K$.

For the second part of the lemma, let $G \cong N \times K$ and let $f_N$ be an $N$-term, and $f_K$ be a $K$-term. Since $G$ is the direct product of $N$ and $K$, every element of $G$ can be written as a unique product $\nu \kappa$ with $\nu \in N$ and $\kappa \in K$. Let $K = \{\kappa_1,\ldots,\kappa_q \}$.  Then we define $f((x_\pi)_{\pi \in G}) = f_K(f_N((x_{\nu  \kappa_1})_{\nu \in N}),\ldots,f_N((x_{\nu  \kappa_q})_{\nu \in N}))$.
It is not hard to see that $f$ is invariant under translating its variables by elements from $N$.  Moreover, for every $\kappa \in K$ we have:
\begin{align*}
f^\kappa((x_\pi)_{\pi \in G}) & = f_K(f_N((x_{\kappa \nu  \kappa_1})_{\nu \in N}), \ldots,f_N((x_{\kappa \nu  \kappa_q})_{\nu \in N}))\\
&= f_K(f_N((x_{ \nu  \kappa \kappa_1})_{\nu \in N}),\ldots,f_N((x_{ \nu \kappa  \kappa_q})_{\nu \in N}))\\
&=(f_K)^\kappa(f_N((x_{\nu \kappa_1})_{\nu \in N}),\ldots,f_N((x_{\nu \kappa_q})_{\nu \in N}))=f((x_\pi)_{\pi \in G}),
\end{align*}
where the third line is equal to the second because the action of $\kappa$ permutes the $\kappa_i$'s and in the same way it permutes the $f_N((x_{\nu\kappa_i})_{\nu\in N})$'s. Finally, the identity in the last line uses that $f_K$ is a $K$-term. Since $N$ and $K$ generate $G$, $f$ is a $G$-term, which concludes the proof.
\end{proof}

Lemma~\ref{lemma:regular} in particular implies the following:

\begin{corollary}\label{cor:product}
Let $G \leq \Sym(G)$ be a regular nilpotent permutation group Then $\Sigma_G$ is equivalent to the join of $\Sigma_{\Z_{p_1}}, \ldots, \Sigma_{\Z_{p_n}}$, where $p_1,\ldots,p_n$ are the prime divisors of $|G|$.
\end{corollary}

\begin{proof}
Since $G$ is nilpotent, it is isomophic to the direct product of its Sylow groups; denote these by $G_{p_1}, G_{p_2},\ldots,G_{p_n}$. By Lemma~\ref{lemma:regular}, $\Sigma_G$ is therefore equivalent to the join of $\Sigma_{G_{p_1}},\ldots,\Sigma_{G_{p_n}}$, where all $G_{p_i}$ are considered as regular permutation groups. In regular permutation groups any nonzero element has no fixpoints. Therefore, by Lemma~\ref{lemma:p-Sylow}, each $\Sigma_{G_{p_i}}$ is equivalent to $\Sigma_{\Z_{p_i}}$, which concludes the proof.
\end{proof}

\section{The local-global property for cyclic loop conditions} \label{sect:example} 

In this section we prove that $\Sigma_G$ has the $n$-local-global property if $G = G_1 \times G_2 \times \cdots \times G_n \leq \Sym(G_1 \cup G_2 \cup \cdots \cup G_n)$ is a direct product of $n$ regular permutation groups.  For regular groups ($n=1$) our proof is based on a simple induction argument that can be best illustrated for cyclic terms:

\begin{example} \label{example:cyclic}
Let $\algA$ be an algebra and let $\overline a\in A^n$. Then $\algA$ satisfies $\Sigma_{\Z_n}$ locally on $\{ \overline a \}$ if there is a constant tuple $(c,c,\ldots,c) \in A^n$ and a term $t \in \Clo(\algA)$ such that 
\[
\begin{pmatrix}
c\\
c\\
\vdots\\
c\\
\end{pmatrix} =  t \begin{pmatrix}
a_1& a_2 &\dots & a_n\\
a_2& a_3 &\dots & a_1\\
\vdots & \vdots &\ddots & \vdots\\
a_n& a_1 &\dots & a_{n-1}\\
\end{pmatrix} = t(M_{\overline a}),
\]
where $M_{\overline a} \in A^{n\times n}$ denotes the matrix whose $(i+1)$-th row is the $i$-th cyclic shift of $\overline a$.

Assume $\algA$ satisfies $\Sigma_{\Z_n}$ locally on all 1-element subsets $\{ \overline a \} \subseteq A^n$. We are then going to show by induction on $k=1,2,\ldots$ that $\algA$ satisfies $\Sigma_{\Z_n}$ also locally on every subset $F \subseteq A^n$ of size $|F|=k$. This will imply that $\Sigma_{\Z_n}$ has the 1-local-global-property.

For $k =1 $ this is trivial. For an induction step $k \to k+1$, let us assume that there is a term $t(x_1,\ldots,x_n)$ that satisfies $\Sigma_{\Z_n}$ on a set $F = \{\overline a^{(1)}, \ldots \overline a^{(k)}\}$. In other words, for every matrix $M_{\overline a^{(i)}}$, there is a constant tuple $\overline c^{(i)} \in A^n$, such that $t(M_{\overline a^{(i)}}) = \overline c^{(i)}$. Our goal is then to construct a term that satisfies $\Sigma_{\Z_n}$ on $F\cup\{\overline b\}$ where $\overline b \in A^n$ is a new tuple.

Let us define $d_i = t(b_i,b_{i+1},\ldots,b_{i-1})$ for all $i = 1,\ldots,n$ and $\overline d = (d_1,\ldots,d_n)$. By our assumptions on $\algA$, there is a term $s \in \Clo(\algA)$ that is cyclic on $\{\overline d\}$, i.e., $s(M_{\overline d}) = (e,e,\ldots,e)$ for some constant $e$. The middle columns in Figure~\ref{fig:matrix} show the images of $M_{\overline a_{(i)}}$ and $M_{\overline b}$ under $t(x_1,x_2,\ldots,x_n)$ and its cyclic shifts, respectively. It is straightforward to see that the term
$$u(x_1,x_2,\ldots,x_n) = s(t(x_1,x_2,\ldots,x_n),t(x_2,x_3,\ldots,x_1), \ldots, t(x_n,x_1,\ldots,x_{n-1})),$$
maps each matrix $M_{\overline a^{(i)}}$ to the constant tuple $s(\overline c^{(i)},\ldots,\overline c^{(i)})$, and $M_{\overline b}$ to $(e,e,\ldots,e)$ (cf. the last column of Figure~\ref{fig:matrix}). Thus $u$ is a cyclic term on $F \cup \{\overline b \}$, which is what we wanted to prove.

\begin{figure}[t]
\FIG{$
\begin{array}{|cccc|cccc|c|}
\hline
 &   &  &   & y_1 & y_2  & & y_n & s(y_1,\ldots,y_n) \\
x_1 & x_2 & \cdots & x_n & t(x_1,\ldots, x_n) & t(x_2,\ldots,x_1) & \cdots & t(x_n,\ldots,x_{n-1}) &  \\
\hline
\multicolumn{4}{|c|}{M_{\overline a^{(1)}}} & \overline c^{(1)} & \overline c^{(1)} & \cdots & \overline c^{(1)} & s(\overline c^{(1)},\ldots,\overline c^{(1)})\\
\multicolumn{4}{|c|}{M_{\overline a^{(2)}}} & \overline c^{(2)} & \overline c^{(2)} & \cdots & \overline c^{(2)} & s(\overline c^{(2)},\ldots,\overline c^{(2)})\\
\multicolumn{4}{|c|}{\vdots} & \vdots & \vdots &  & \vdots & \vdots\\
\multicolumn{4}{|c|}{M_{\overline a^{(k)}}} & \overline c^{(k)} & \overline c^{(k)} & \cdots & \overline c^{(k)} & s(\overline c^{(k)},\ldots,\overline c^{(k)})\\
\hline
b_1& b_2 &\dots & b_n & d_1 & d_2 & \cdots & d_n & e \\
b_2& b_3 &\dots & b_1 & d_2 & d_3 & \cdots & d_1 & e \\
\vdots & \vdots  &  & \vdots  & \vdots  & \vdots  &  & \vdots  & \vdots \\
b_n& b_1 &\dots & d_{n-1} & d_n & d_1 & \cdots & d_{n-1} & e \\
\hline
\end{array}
$}{\caption{Construction of a local $n$-cyclic term on $F \cup \{\overline b\}$}}
\label{fig:matrix}
\end{figure}
\end{example}

In Lemma \ref{lemma:lg-regular} we continue by generalizing Example \ref{example:cyclic} to all regular permutation groups. For a more compact presentation we refrain from illustrating the proof of Lemma \ref{lemma:lg-regular} (and subsequently Theorem \ref{theorem:prod-regular}) by matrices as in Figure \ref{fig:matrix}, but we invite the reader to keep similar pictures in mind.

\begin{lemma} \label{lemma:lg-regular}
Let $G \leq \Sym(G)$ be a regular permutation group. Then $\Sigma_G$ has the $1$-local-global property.\end{lemma}

\begin{proof}
Let $\algA$ be an algebra. Recall that a term $t((x_\pi)_{\pi \in G}) \in \Clo(\algA)$ is a $G$-term on a set $F\subseteq A^G$ if it satisfies the equations $t((x_\pi)_{\pi \in G}) = t^{\psi}((x_\pi)_{\pi \in G}) = t((x_{\psi \circ \pi})_{ \pi \in G})$ for all $\psi \in G$ and all $\overline x \in F$. Let us assume that $\algA$ has $G$-terms on all subsets of $A^G$ of size 1. We then show by induction on $|F|$ that $\algA$ has $G$-terms on all finite subsets $F \subseteq A^G$. 

For the induction step, let $t$ be a $G$-term on $F \subseteq A^G$; our goal is to construct a $G$-term on a 1-element extension $F \cup \{ \overline b\}$ of $F$ where $\overline b\in A^G$ is a new element. Let us define the tuple $\overline d = (d_\pi)_{\pi \in G} \in A^{G}$ by $d_\pi = t^\pi(\overline b)$ for every $\pi \in G$. By our assumptions on $\algA$, there is a term $s \in \Clo(\algA)$ such that $s(\overline d ) = s^{\psi}(\overline d)$, for every $\psi \in G$. We then define $u(\overline x) = s((t^{\pi}(\overline x))_{\pi \in G})$.

By our hypothesis on $F$, for each $\overline a\in F$ there is a constant tuple $c_{\overline a}$ such that $t^{\psi}(\overline a) = c_{\overline a}$ for all $\psi \in G$. This implies that $u^{\psi}(\overline a) = s( c_{\overline a},\ldots, c_{\overline a}) = u(\overline a)$ for all $\psi \in G$, hence $u$ is a $G$-term on $F$. For the new tuple $\overline b$ and any $\psi \in G$ we have
\begin{align*}
u^\psi(\overline b) = u(\overline b^\psi) &= s((t^{\pi}(\overline b^{\psi})_{\pi \in G}) = s((t^{\psi \circ \pi}(\overline b))_{\pi \in G}) \\
&= s((d_{\psi\circ \pi})_{\pi \in G}) =s(\overline d^\psi) = s^{\psi}(\overline d ) = s(\overline d) = u(\overline b).
\end{align*}
Thus $u$ is a $G$-term on $F \cup \{ \overline b\}$. This finishes the proof of the induction step, hence $\Sigma_G$ has the local-global property.

\end{proof}

As a direct corollary we get:

\begin{corollary} \
\begin{enumerate}
\item  $\Decide(\Sigma_G)$ can be decided in time $\mathcal O(\arity(\algA) \| \algA \|^{2|G|})$ for every regular group $G \leq \Sym(G)$.  
\item For a fixed $n\in \N$, the existence of an $n$-ary cyclic term, $\Decide(\Sigma_{\Z_n})$, can be decided in time $\mathcal O(\arity(\algA) \| \algA \|^{2 p})$, where $p$ is the largest prime divisor of $n$.
\end{enumerate}
\end{corollary}

\begin{proof}
The result about regular groups $G$ follows immediately from Lemmas~\ref{lemma:local-global} and \ref{lemma:lg-regular}.

For an $n$-ary cyclic term, we know from Corollary~\ref{cor:product} that $\Sigma_{\Z_n}$ is equivalent to the satisfaction of all of $\Sigma_{\Z_{p_1}}, \ldots, \Sigma_{\Z_{p_k}}$, where $p_1,\ldots,p_k$ are the prime divisors of $n$. Since all groups $\Z_{p_i}$ are regular , $\Decide(\Sigma_{\Z_{p_i}})$ can be tested in time $\mathcal O( \arity(\algA) \| \algA \|^{2 p_i})$. Since $k$ is a constant, 
$\mathcal O(\sum_{i=1}^k \arity(\algA) \| \algA \|^{2 p_i})=\mathcal O(\arity(\algA) \| \algA \|^{2 p})$, giving us the result.
\end{proof}

We next generalise Lemma \ref{lemma:lg-regular} to direct products of regular groups. Although the core idea is the same,  the proof is much more technical, as it involves a nested induction with three layers where the innermost layer is proved by an induction-like contradiction.

\begin{theorem} \label{theorem:prod-regular}
Let $G_1, \ldots, G_n$ be regular permutation groups and let 
\[
G = G_1 \times \cdots \times G_n \leq \Sym(G_1 \cup \cdots \cup G_n)
\] be their direct product with the natural action on the disjoint union $G_1 \cup \cdots \cup G_n$. Then $\Sigma_G$ has the $n$-local-global property.
\end{theorem}


\begin{proof}

Let $\algA$ be an algebra with the universe $A$. Given a tuple $\overline a \in A^{G_1 \cup \cdots \cup G_n}$, we will denote by $\overline a^{(i)} \in A^{G_i}$ its projections to the coordinates labelled by $G_i$. For any subset $I \subseteq [n]$, let us define 
\[
O_I = \{ \overline a = (\overline a^{(1)},\ldots, \overline a^{(n)}) \in A^{G_1 \cup \cdots \cup G_n} \, \colon \, \forall i \in I,\, \overline a^{(i)} \text{ is constant}\}. 
\]

For any $m=0,1,\ldots,n$, let $S(m)$ be the following statement: \\

$S(m)$: ``For every sequence of pairwise different indices $\overline i = (i_1,i_2,\ldots,i_{m}) \in [n]^m$, for every set $F = \{ \overline a_0, \overline a_1, \overline a_2, \ldots, \overline a_{m-1} \} \subseteq A^{G_1 \cup \cdots \cup G_n}$ such that $\overline a_j \in O_{\{i_1,i_2,\ldots,i_j\}}$ and for every finite subset $M \subseteq O_{\{i_1,i_2,\ldots,i_m\}}$, there is term $t \in \Clo(\algA)$ that is a $G$-term on $F \cup M$.''\\

Note that $\algA \models S(n)$ if and only if $\algA$ satisfies $\Sigma_G$ locally on all sets of the form $F = \{\overline a_0 \ldots, \overline a_{n-1} \} \subseteq A^{G_1 \cup \cdots \cup G_n}$ with $\overline a_j \in O_{\{i_1,i_2,\ldots,i_j\}}$. If in particular $\algA$ satisfies $\Sigma_G$ locally on all sets $F$ of size $n$, then $S(n)$ holds. On the other hand $\algA \models S(0)$ states that $\algA$ has a $G$-term on \emph{all} finite subsets $M \subseteq O_\emptyset =  A^{G_1 \cup \cdots \cup G_n}$. Thus if we can prove that $\algA \models S(n)$ implies $\algA \models S(0)$, it will follow that $\Sigma_G$ has the $n$-local-global property and we will be done.

So let us assume that $\algA \models S(n)$. We are going to prove by induction that then also $\algA \models S(m)$ for all $m = n,n-1,\ldots,0$. For $m = n$ this is trivial. So let us consider the induction step $m+1 \to m$.

Assume without loss of generality that $\overline i = (1,2,\ldots,m)$ (otherwise we reorder the direct factors of $G$). Thus our goal is to prove that for every set $F = \{\overline a_0, \overline a_1, \ldots, \overline a_{m-1}\}$ such that $\overline a_i \in O_{[i]}$ and every finite $M \subseteq O_{[m]}$ there is a $G$-term $f$ on $F \cup M$. We will prove this by induction on $|M|$.

For $M = \emptyset$ this is true since $\algA \models S(n)$. For an induction step, we are going to show that whenever there is a $G$-term $t$ on $F \cup M$ for a finite set $M \subseteq O_{[m]}$, then there is also a $G$-term on $F \cup M \cup \{ \overline b\}$ for a new tuple $\overline b \in O_{[m]}$. In order to prove this, let $T(j)$ be the following statement for $j \in \{1,2,\ldots,n\}$ (and for fixed $F,M,\overline b$):\\

$T(j)$: ``There is a $t \in \Clo(\algA)$, which is a $G$-term on $F\cup M$ and additionally satisfies $t^{\alpha\beta}(\overline b) = t^\beta(\overline b)$ for all $\alpha \in G_1 \times G_2 \times \cdots \times G_j$ and all $\beta \in G_{j+1}$.''\\

For $j=n$ we interpret the above equation as $t^{\alpha}(\overline b) = t(\overline b)$, for all $\alpha \in G_1 \times \cdots \times G_n$; so $T(n)$ simply states that $t$ is a $G$-term on $F\cup M \cup \{\overline b \}$.

Note for future reference that if $t$ is a witness for $T(j)$ and $\theta\in G_{j+1}$ then $t^\theta$ is also a witness for $T(j)$: First, $t^\theta$ is a $G$-term since $t$ is a $G$-term and $\theta\in G$. Second, 
\[
(t^\theta)^{\alpha\beta}(\overline b)=t^{\alpha\beta\theta}(\overline b)=
t^{\alpha(\beta\theta)}(\overline b)=
t^{\beta\theta}(\overline b)=(t^\theta)^\beta(\overline b),
\]
where in the middle of the chain of equalities we used $T(j)$ on $t$ since $\beta\theta\in G_{j+1}$.

Note that $\algA \models T(j)$ for $j = 0,1,\ldots,m$ since $\overline b \in O_{[m]}$. We are going to show by induction on $j = m,m+1,\ldots,n$ that $\algA \models T(j)$ for all $j$. (This is the third level of induction in our proof. Since $\algA \models T(n)$ is equivalent with $\algA$ having a $G$-term on $F\cup M \cup \{\overline b\}$, once we finish the induction on the third level, we will be able to increase $|M|$, which will give us the induction step on the second level of induction.)

For the induction step $T(j-1)\Rightarrow T(j)$ we first define 
\[
\Phi_t = \{ \beta \in G_{j+1} \colon t^{\alpha\beta}(\overline b) = t^\beta(\overline b) \text{ for all } \alpha \in G_1 \times G_2 \times \cdots \times G_j \}.
\]

Choose a $t$ from among the terms that witness $T(j-1)$ so that  $\Phi_t$ is of maximal cardinality. Clearly, if $\Phi_t = G_{j+1}$, then $T(j)$ holds, and we are done.

Suppose for a contradiction that there is a $\theta \in G_{j+1} \setminus \Phi_t$. We can assume without loss of generality that $\theta = id$ is the identity; otherwise we take $t^{\theta}$ and $\Phi_t \circ \theta^{-1}$ instead of $t$ and $\Phi_t$.  

For any tuple $\overline c \in A^{G_1 \cup \cdots \cup G_n}$, let for short $\overline c'  \in A^{G_1 \cup \cdots \cup G_n}$ denote the tuple defined by
\[
\overline c_\pi' = \begin{cases} t^{\pi}(\overline c) \text{ if } \pi \in G_j,\\
c_\pi \text{ else.}
\end{cases}
\]
Since $t$ is a $G$-term on $F$, we get $\overline a_i' \in O_{[i] \cup \{j\}}$ for every $\overline a_i \in F$. Similarly, $t$ is a $G$-term on $M$, so $\overline d' \in O_{[m] \cup \{j\}}$ for every $\overline d \in M$. 

Furthermore, for all $\alpha \in G_1 \times \cdots \times G_j$, all $\pi\in G_j$, and all $\beta \in \Phi_t$ we have (note that $\beta$ commutes with $\pi$ and that $\alpha\pi\in G_1\times\cdots\times G_j$):
\[t^\pi(\overline b^{\alpha \beta})
=t^{\alpha\beta\pi}(\overline b)=t^{(\alpha\pi)\beta}(\overline b)=t^{\beta}(\overline b).
\]
Thus the value of $t^\pi(\overline b^{\alpha \beta})$ does not depend on $\pi$ at all. Therefore,
$(\overline b^{\alpha \beta})' \in O_{[m] \cup \{j\}}$ 
for all $\alpha\in G_1 \times \cdots \times G_j$ and all $\beta\in \Phi_t$. 

We define $F' = \{ \overline b', \overline a_0',\overline a_{1}', \ldots, \overline a_m'\}$ and 
\[
M' = \{ \overline d' \colon \overline d \in M\} \cup \{ (\overline b^{\alpha \beta})' \colon \alpha \in G_1 \times \cdots \times G_j, \beta \in \Phi_t \}.
\]
By the above arguments, $M'\subseteq O_{[m]\cup \{j\}}$.
Recall that $\algA \models S(m+1)$ by the induction hypothesis for the outermost induction; applying $S(m+1)$ with the indices $\overline i =(j,1,2,\dots,m)$ and sets $F'$ and $M'$ we see that there is a term $t'$, which is a $G$-term on $F' \cup M'$.

We then define the term $s$ as follows:
$$s(\overline x) = t'(\overline x^{(1)},\ldots, \overline x^{(j-1)}, (t^{\pi}(\overline x))_{\pi \in G_j}, \overline x^{(j+1)}, \ldots, \overline x^{(n)}).$$

Observe that for any $\overline c\in A^{G_1 \cup \cdots \cup G_n}$ the value of $s(\overline c)$ equals $t'(\overline c')$. From this and the fact that $t'$ is a $G$-term on $F'\cup M'$, it follows that $s$  is a $G$-term on $F\cup M$. 

We now claim that $\Phi_s \supseteq \Phi_t \cup \{ \id \}$. Note that $\id \in \Phi_s$ implies in particular that $s$ satisfies $T(j-1)$. Thus, if we manage to prove this claim, it will contradict the maximality of $\Phi_t$.

In order to see that $\id \in \Phi_s$, let $\alpha \in G_1 \times \cdots \times G_{j}$ with the decomposition $\alpha = \gamma \delta$ into $\gamma \in G_1 \times \cdots \times G_{j-1}$ and $\delta \in G_j$. Then 
\begin{align}
s^{\alpha}(\overline b)&=s(\overline b^\alpha)=s(\overline b^{\gamma\delta})\\ &= t'( (\overline b^{(1)}, \ldots, \overline b^{(j-1)})^{\gamma}, (t^{\pi}(\overline b^{\gamma \delta}))_{\pi \in G_j}, \overline b^{(j+1)},\ldots, \overline b^{(n)} )\\
&= t'( (\overline b^{(1)}, \ldots, \overline b^{(j-1)})^{\gamma}, (t^{\gamma  \delta  \pi}(\overline b))_{\pi \in G_j}, \overline b^{(j+1)},\ldots, \overline b^{(n)})\\
&= t'( (\overline b^{(1)}, \ldots, \overline b^{(j-1)})^{\gamma}, (t^{\delta  \pi}(\overline b))_{\pi \in G_j}, \overline b^{(j+1)},\ldots, \overline b^{(n)}) \label{eq:prod-regular1} \\
&= t'((\overline b')^{\gamma  \delta})= t'^{\gamma  \delta}(\overline b') = t'(\overline b') = s(\overline b). \label{eq:prod-regular2}
\end{align}
In the above, equation (\ref{eq:prod-regular1}) holds since $t$ satisfies $T(j-1)$; the equations in line (\ref{eq:prod-regular2}) hold, since $t'$ is a $G$-term for $\overline b'$.

In order to see that $\Phi_s \supseteq \Phi_t$, let $\alpha \in G_1 \times \cdots \times G_{j}$, and $\beta \in \Phi_t$ (and hence $\beta\in G_{j+1}$). As above, let $\alpha = \gamma  \delta$ be the unique decomposition of $\alpha$ into $\gamma \in G_1 \times \cdots \times G_{j-1}$ and $\delta \in G_j$. Then (note that $\beta$ commutes with both $\pi$ and $\alpha$)
\begin{align}
s^{\alpha\beta}(\overline b)&=s(\overline b^{\beta\alpha}) \\
&= t'( (\overline b^{(1)}, \ldots, \overline b^{(j-1)})^{\gamma}, (t^{\pi}(\overline b^{\beta\alpha}))_{\pi \in G_j}, (\overline b^{(j+1)})^{\beta},\overline b^{(j+2)}, \ldots, \overline b^{(n)} )\\
&= t'( (\overline b^{(1)}, \ldots, \overline b^{(j-1)})^{\gamma}, (t^{\beta\alpha\pi}(\overline b))_{\pi \in G_j}, (\overline b^{(j+1)})^{\beta},\overline b^{(j+2)}, \ldots, \overline b^{(n)})\\
&= t'( (\overline b^{(1)}, \ldots, \overline b^{(j-1)})^{\gamma}, (t^{\alpha\pi\beta}(\overline b))_{\pi \in G_j}, (\overline b^{(j+1)})^{\beta},\overline b^{(j+2)}, \ldots, \overline b^{(n)})\\
&= t'( (\overline b^{(1)}, \ldots, \overline b^{(j-1)})^{\gamma}, (t(\overline b^{\beta}))_{\pi \in G_j}, (\overline b^{(j+1)})^{\beta},\overline b^{(j+2)}, \ldots,\ldots, \overline b^{(n)})  \label{eq:prod-regular3} \\
&= (t')^{\gamma}((\overline b^\beta)') = t'((\overline b^\beta)') = s(\overline b^{\beta})=s^\beta(\overline{b}).\label{eq:prod-regular4}
\end{align}

In the above, equation (\ref{eq:prod-regular3}) holds since $\beta \in \Phi_t$; the middle equation in (\ref{eq:prod-regular4}) holds because $t'$ is a $G$-term for $(\overline b^\beta)'\in M'$. This completes the proof that $\Phi_s \supseteq \Phi_t \cup \{ \id \}$, which contradicts the maximality of $t$.

Thus we showed that $T(j-1)$ implies $T(j)$. By induction on $j=0,1,\ldots,n$, we obtain $\algA \models T(n)$, i.e. there is a $G$-term on $F\cup M \cup \{\overline b \}$. This in turn finishes the proof of $\algA \models S(m)$. Downwards induction on $m = n,\ldots,1,0$ implies that $\algA \models S(0)$. 

In conclusion, we have proved that when $\algA$ is an algebra such that $\algA \models S(n)$, then $\algA \models S(0)$. This implies that $\Sigma_G$ has the $n$-local-global property.

\end{proof}

As an immediate corollary of Theorem \ref{theorem:prod-regular} we obtain:

\begin{corollary} \
\begin{itemize}
\item For every direct product of regular groups $G = G_1 \times \cdots \times G_n$, $\Decide(\Sigma_G)$ can be solved in time $\mathcal O(\arity(\algA) \| \algA \|^{n (|G| + \sum_{i=1}^n |G_i|)})$. 
\item For every cyclic loop condition $\Sigma_{\langle \pi \rangle}$, $\Decide(\Sigma_{\langle \pi \rangle})$ is in $\comP$.
\end{itemize}
\end{corollary}

\begin{proof}
The result about regular groups $G$ follows immediately from Theorem \ref{theorem:prod-regular} and Lemma \ref{lemma:local-global}. Further, recall from Theorem \ref{theorem:loop-condition} that every cyclic loop condition $\Sigma_{\langle \pi \rangle}$ is equivalent to the join of cyclic loop conditions $\Sigma_{\langle \pi_1 \rangle}, \ldots, \Sigma_{\langle \pi_n \rangle}$, such that every $\pi_i$ has orbits of distinct prime length. Every group $\langle \pi_i \rangle$ is the direct product of prime cyclic groups, and therefore $\Decide(\Sigma_{\pi_i}) \in \comP$ for every $i = 1,\ldots,n$. This implies that also $\Decide(\Sigma_{\langle \pi \rangle}) \in \comP$.
\end{proof}

\section{Failure of local-global} \label{sect:failure} 

In this section we prove the failure of the local-global property for several $G$-term conditions:

\begin{theorem} \label{theorem:failureltg}
Let $G \leq \Sym(n)$ be a permutation group with no fixpoints and assume that there is a permutation $\alpha \in G$ such that $\alpha$ has exactly one fixpoint and all other orbits of $\alpha$ have the same size. Then for every $k \in \N$ there is an algebra $\algA_k$ such that:
\begin{itemize}
\item $\algA_k = (A_k, f_0,\ldots, f_k)$ with $(k+1)n<|A_k| < (k + 2) n$,
and $f_0,\ldots, f_k$ are $n$-ary,
\item $\algA_k$ is idempotent,
\item $\algA_k$ satisfies $\Sigma_{\Sym(n)}$ on every $F\subseteq A_k^n$ with $|F|=k$,
\item $\algA_k$ does not satisfy $\Sigma_G$.
\end{itemize}
\end{theorem}

\begin{proof}
Note that $n>1$ since $\Sym(1)$ has a fixpoint.
Let $N = \{1,\ldots,n\}$, and let us regard $G$ as a permutation group on $N$. Without loss of generality, let $1$ be the fixpoint of $\alpha$. We can further assume that the length of all other orbits of $\alpha$ is a prime $p$; if not, we take a suitable power of $\alpha$ instead. Note that $n \equiv 1 \bmod p$.

Let us now consider the component-wise action of $\langle \alpha \rangle$ on tuples from $N^\ell$ for $\ell\in \N$. The only fixpoint of this action is $(1,1,\ldots,1)$; all other orbits have size $p$. In particular all orbits of non-constant tuples in $N^\ell$ have size $p$. Let $T_\ell$ be a transversal set of these orbits and let $T = \bigcup_{\ell \in \N} T_\ell$. Then every tuple $\overline x \in \bigcup_{\ell \in \N} N^\ell$ has a unique representation as $\overline x = \alpha^i (\overline t)$, with $\overline t \in T$ and $0 \leq i \leq p-1$.

We define the universe of $\algA_k$ to be the disjoint union of $\{0,1,\ldots,k\} \times N$ and $\Z_p$. For any tuple $\overline x$ over $\algA_k$, let $\overline x_N$ denote the tuple consisting only of its entries from $\{0,1,\ldots,k\} \times N$.

We then define the operations $f_i$ for $i=0,1,\ldots,k$ by the following rules:

\begin{enumerate}[(1)]
\item\label{itm:allZp} $f_i(x_1,\ldots,x_n) = \sum_{j=1}^n x_j$ if $x_j \in \Z_p$ for all $j = 1,\ldots,n$ (addition is modulo $p$ here),
\item\label{itm:constant} $f_i(x_1,\ldots,x_n) = c$ if $\overline x_N = (c,c,\ldots,c)$ for some $c\in A_k$,
\item\label{itm:orbit} $f_i(x_1,\ldots,x_n) = j \in \Z_p$, if $\overline x_N = ((i,a_1),\ldots,(i,a_m))$ and $(a_1,\ldots,a_m) = \alpha^j (\overline t)$ for some $\overline t \in T_m$,
\item\label{itm:else} $f_i(x_1,\ldots,x_n) = 0 \in \Z_p$ else.
\end{enumerate}

Each operation $f_i$ is idempotent by~(\ref{itm:allZp}) and (\ref{itm:constant}). Further note that $f_i$ is symmetric on all tuples from $A_k$ except those satisfying the hypothesis of rule~(\ref{itm:orbit}). When rule~(\ref{itm:orbit}) applies,  $\overline x_N$ only contains tuples from $\{i\} \times N$, and $f_i(\overline x)$ counts the number of times one needs to apply $\alpha$ to reach $\overline x_N$ from the transversal set $T_m$. 

We first show that $\algA_k$ has a symmetric term on every family of $k$ tuples. Choose any $\overline a_1, \ldots, \overline a_k \in A_k^n$. Each tuple $\overline a_j$ can satisfy the hypothesis of rule~(\ref{itm:orbit}) for at most one basic operation. Therefore we are left with at least one $i \in \{0,1,\ldots,k\}$ such that no tuple $\overline a_j$ satisfies condition~(\ref{itm:orbit}) with respect to $f_i$. This operation $f_i$ is a symmetric term on $\{\overline a_1,\dots,\overline a_k\}$.

It only remains to prove that $\algA_k$ has no (global) $G$-terms. Observe first that $\Z_p$ is a subuniverse of $\algA_k$. When restricted to $\Z_p$, all the operations $f_1,\dots,f_k$ are equal to $f_j(x_1,\ldots,x_n) = \sum_{i=1}^n x_i$ which, as an idempotent linear map, preserves all affine subspaces of powers of $(\Z_p,+)$. In particular, the operations of $\algA_k$ preserve the relation $R \subseteq \Z_p^I$, with $I = \{0,\ldots,k\} \times \Z_p$, which is defined by the linear identities
\[
(y_{i,j})_{(i,j) \in I} \in R \Leftrightarrow \forall j \in \Z_p \colon \sum_{\ell=0}^{k} y_{\ell,j+1} =1+ \sum_{\ell=0}^{k} y_{\ell,j}.
\]
We use the non-standard labelling of the coordinates of $R$ by $I$ (instead of natural numbers), in order to simplify the presentation of our proof. Note that $R$ is not empty, but there is no $\overline y \in R$ such that $y_{\ell,0} = y_{\ell,1} = \cdots = y_{\ell,p-1}$ for every $\ell = 0,\ldots, k$.

Next, let us define the tuple $\overline {q(a)} \in A^I$ for each $a \in N$ by $q(a)_{i,j} = (i,\alpha^{j}(a))$ for all $i \in \{0,1,\ldots,k\}$, $j \in \Z_p$. 
A key observation we will need later is that $\overline{q(\alpha(a))}$ is $\overline{q(a)}$ where the indices in each of the coordinate blocks $\{i\}\times \Z_p$ have been cyclically shifted. To be more specific,  $\overline{q(\alpha(a))}_{i,j}  = (i,\alpha^{j+1}(a)) = \overline{q(a)}_{i,j+1}$, for all $(i,j) \in I$. 

Let us define the relation $Q = \{ \overline{q(a)} \colon a \in N \} \cup R \subseteq A_k^{I}$. We shall show that $Q$ is invariant under $\algA_k$. We will prove that $Q$ is preserved by the basic operation $f_0$; the argument for the other operations $f_i$ is analogous. Let $\overline q_1,\ldots, \overline q_n \in Q$; we need to show $f_0(\overline q_1,\ldots, \overline q_n) \in Q$. We will consider several cases.

If all $\overline q_1,\ldots, \overline q_n$ are elements of $R$, then also $f_0(\overline q_1,\ldots, \overline q_n) = \sum_{i=1}^n \overline q_i \in R$, since $R$ is an affine subspace of $\Z_p^{I}$.

In the remaining cases $f_0(\overline q_1,\ldots, \overline q_n)$ only depends on the tuples that are not in $R$. Without loss of generality let these be the first $r$ tuples $\overline{q(a_1)},\ldots, \overline{q(a_r)}$. If all of these tuples are equal to the same $\overline{q(a)}$, we get $f_0(\overline q_1,\ldots, \overline q_n) = \overline{ q(a)} \in Q$ by rule~(\ref{itm:constant}), so let us assume that these tuples are not all the same.

Denote by $M \in A^{I \times r}$ the matrix with columns $\overline{q{(a_1)}},\ldots, \overline{q{(a_r)}}$ \footnote{So the rows of $M$ are indexed by $I$, \emph{not} natural numbers}. The
``$(i,j)$-th'' row of the matrix $M$ is of the form
$((i,\alpha^j({a_1})),(i,\alpha^j({a_2})),\ldots,(i,\alpha^j({a_r}))$.
Let $\overline t \in T_r$ and $\ell \in \Z_p$ be such that $(a_1,\ldots,a_r) = \alpha^\ell(\overline t)$. Then it follows from rule~(\ref{itm:orbit}) that 
$f_0(\overline q_1,\ldots, \overline q_n)_{(0,j)} = \ell+j$. By rule~(\ref{itm:else}),  $f_0(\overline q_1,\ldots, \overline q_n)_{(i,j)} = 0$ for $i \neq 0$. Therefore $f_0(\overline q_1,\ldots, \overline q_n) \in R \subseteq Q$. 
Thus $Q$ is preserved by $\algA_k$.

Suppose now for a contradiction that there is a (global) $G$-term $t \in \Clo(\algA_k)$. In particular $t$ needs to satisfy $t \approx t^{\alpha}$, and therefore
\begin{align*}
t(\overline{q (1)}, \overline{ q(2)},\ldots, \overline{q(n)}) &=
t^\alpha(\overline{q (1)}, \overline{ q(2)},\ldots, \overline{q(n)})\\ 
&=
t(\overline{q(\alpha(1))}, \overline{q(\alpha(2))},\ldots, \overline{ q(\alpha(n))})
\end{align*}
Denote $t(\overline{q (1)}, \overline{ q(2)},\ldots, \overline{q(n)})$ by $\overline r$. Since $Q$ is preserved by $t$, we have $\overline r\in Q$.

Let $M$ be the matrix with columns $\overline{q(1)}$, $\overline{q(2)}$, \dots, $\overline{q(n)}$ (in this order) and let $M'$ be the matrix with columns $\overline{q(\alpha(1))}$, $\overline{q(\alpha(2))}$, \dots, $\overline{q(\alpha(n))}$. Recall that $\overline{q(\alpha(j))}$ is $\overline{q(j)}$ where for each $i$ the coordinates $\{i\}\times \Z_p$ are cyclically shifted. 

From this it follows that the $(i,j)$-th row of $M'$ is the $(i,j+1)$-th row of $M$.
By the identity $\overline r=t(M)=t(M')$, we obtain $r_{i,j}=r_{i,j+1}$ for each applicable $i,j$, so $r_{i,0} = r_{i,1} = \cdots = r_{i,p-1}$ for every $i = 0,\ldots, k$. Therefore $\overline r \notin R$.

As noted above, $\overline r\in Q$. Since $\overline r\not\in R$, we must have $\overline r=\overline{q(a)}$ for some $a\in N$. Since the value of $(i,\alpha^j(a))$ must not depend on $j$, $a$ needs to be a fixpoint of $\alpha$, i.e., $\overline r=\overline{q(1)}$.

However, since $G$ has no fixpoints, there is a permutation $\pi\in G$ that sends 1 to some $\pi(1)\neq 1$.  Then $\alpha' = \pi  \alpha  \pi^{-1} \in G$ has $\pi(1)$ as a fixpoint; all the other orbits of $\alpha'$ have the same size $p$. Applying the same argument to $\alpha'$ instead of $\alpha$, we arrive at $\overline r=\overline{q(\pi(1))}$, which is in contradiction to $\pi(1) \neq 1$. This finishes our proof.
\end{proof}

As a direct corollary we obtain failure of the local-global property for several conditions:

\begin{corollary} \label{corollary:failure}
The Maltsev condition $\Sigma_G$ does not have the local-global property for idempotent algebras for:
\begin{enumerate}
\item the symmetric group $G = \Sym(n)$ for any $n \geq 3$,
\item\label{itm:dihedral} the dihedral groups $G = D_n \leq \Sym(n)$ for all odd $n>1$,
\item the alternating groups $G = A_n \leq \Sym(n)$ for all $n \geq 4$, and
\item the block-symmetric group $G=  \Sym(n) \times \Sym(n+1)$ for every $n \geq 2$.
\end{enumerate}
\end{corollary}

\begin{proof}
It is easy to see that all of the permutation groups in the list have no fixpoints. Therefore we only need to find a permutation $\alpha$ satisfying the properties in Theorem \ref{theorem:failureltg}.
\begin{enumerate}
\item For $G = \Sym(n)$, we set $\alpha = (23\cdots n)$
\item For $G = D_n$, let $\alpha = (2,3)(4,5), \cdots (n-1, n)$
\item For $G = A_n \leq \Sym(n)$. If $n$ is even, we set $\alpha = (23\cdots n)$; if $n = 2k +1$, we set $\alpha = (1,\ldots,k)(k+1,\ldots,2k)$.
\item For $G= \Sym(n) \times \Sym(n+1)$ (with $\Sym(n)$ acting on $[n]$ and $\Sym(n+1)$ acting on $\{n+1,n+2,\dots, 2n+1\}$), we let $\alpha = (1,\ldots,n)(n+1,\ldots 2n)$.
\end{enumerate}
\end{proof}

\section{The local-global property for oligomorphic algebras}  \label{sect:oligomorphic}
For algebras with infinite universe it is in general not true that a Maltsev condition $\Sigma$ is satisfied if it is satisfied on all subsets of finite size even if $\Sigma$ has the local-global property in the sense of Definition \ref{definition:local-global}. A counterexample can be found in \cite{kazda-decidingQWNU}, where it is shown that there is a countable idempotent algebra that has quasi WNU terms on all finite subsets, but not on the full universe.

However, there are additional ``finiteness'' conditions on an algebra or clone that sometimes allow us to lift local properties to the full universe. A clone $\mathcal C \subseteq \bigcup_{n \in \N}A^{A^n}$ on a countablly infinite set $A$ is called \emph{oligomorphic} if the group of its unary invertible elements $\mathcal C^{inv} = \{ f \in \mathcal C \colon f \colon A \to A $ is bijective $ \}$ is an \emph{oligomorphic} permutation group, meaning that the action of $\mathcal C^{inv}$ on finite powers $A^n$ has only finitely many orbits for every $n \in \N$.

Clones on an infinite set $A$ come with a natural topology, the \emph{topology of pointwise convergence}, in which a series of operations has a limit $(f_n)_{n \in \N} \to f$, if and only if all $f_n$ and $f$ have the same arity $k$, and for every finite subset $F \subseteq A^k$: $f|_F = f_n|_F$ for all big enough $n$. Oligomorphic clones that are closed with respect to the topology of pointwise convergence play a central role in the study of infinite domain constraint satisfaction problems as they are the polymorphism clones of $\omega$-categorical structures; we refer to \cite{Bodirsky-book} for further background.

We are going to show that for any condition $\Sigma_G$ which has the $k$-local-global property (in the sense of Definition \ref{definition:local-global}) a closed oligomorphic clone $\mathcal C$ satisfies $\Sigma_G$ if and only if $\mathcal C$ satisfies $\Sigma_{G}$ on every subset $F \in A^n$ with $|F|=k$.

\begin{theorem} \label{theorem:oligomorphic}
Let $G \leq \Sym(n)$ be a finite permutation group such that $\Sigma_G$ has the $k$-local-global property, and let $\mathcal C$ be a closed oligomorphic clone on a countable set $A$. Then $\mathcal C \models \Sigma_G$ if and only if $\mathcal C$ satisfies $\Sigma_{G}$ on every subset $F \in A^n$ with $|F|=k$.
\end{theorem}

\begin{proof}
Let $\mathcal C$ be a closed oligomorphic clone that satisfies $\Sigma_{G}$ on every subset $F \subseteq A^n$ with $|F|=k$. Since $\Sigma_G$ has the $k$-local-global property, we know that for \emph{every} finite $F \subseteq A^n$, there is an operation $f_F(x_1,\ldots,x_n) \in \mathcal C$, such that $f_F$ is a $G$-term on $F$.

Let $\overline a_1,\overline a_2, \ldots$ be an enumeration of $A^n$. We then define a pre-order on all pairs $(i ,f)$, such that $i \in \N$ and $f$ is a $G$-term on $F_i = \{\overline a_1,\overline a_2,\ldots, \overline a_i\}$ by setting $(i,f) \leq (j,g)$ if $i \leq j$ and there is a $u \in \mathcal C^{inv}$, such that $u \circ f |_{F_i} = g |_{F_i}$. Let $\sim$ be the equivalence given by the preorder $\leq$. Since $u$ is invertible, $(i,c)$ and $(j,d)$ are $\sim$-equivalent if and only if $(i,c)\leq (j,d)$ and $i=j$.

Since $C^{inv}$ is oligomorphic, for a fixed $i$ there are only finitely many $i$-elements sets $\{ b_1,\dots, b_i\} \subseteq A$ modulo $C^{inv}$. In particular, this implies for a fixed $i$ there are only finitely many equivalence classes $[(i,c)]_\sim$. This implies that the graph of $\leq$ modulo $\sim$ is a tree. For $i\in \N$ the number of vertices  $[(i,c)]_\sim$ on the $i$-th level of this tree is finite (and positive). In other words, we have an infinite finitely branching tree, which by K\"onig's lemma has an infinite branch.

Thus there is a sequence $f_1, f_2, f_3, \ldots$ of operations in $\mathcal C$, such that for every $i \in \N$ $f_i$ is a $G$-term on $F_i$ and there is a $u_i \in \mathcal C^{inv}$ such that $u_i \circ f_{i+1} |_{F_i} = f_i |_{F_i}$. Without loss of generality we can assume that $u_i = id$ otherwise we substitute $f_{i+1}$ with $u_i^{-1} f_{i+1}$, which is still a $G$-term on $F_{i+1}$.

Since the sequence $f_1, f_2, f_3, \ldots$ is eventually constant on every finite subset $F \subseteq A^n$, it has a limit $f$ with respect to the topology of pointwise convergence. 
As the limit of all $f_i$, the operation $f$ is a $G$-term on $A^n$. Further $f \in \mathcal C$, since $\mathcal C$ is a closed clone. This finishes the proof. 
\end{proof}

We remark that our proof of Theorem \ref{theorem:oligomorphic} can be adapted to other height 1 Maltsev conditions that have the local-global property. Similar arguments using K\"onig's lemma are sometimes also referred to as \emph{(standard) compactness argument} in the literature, a prominent example is the `Lift Lemma' in Chapter 10 of \cite{Bodirsky-book}.

\section{Discussion and open problems} \label{sect:openproblems}
Our results in Theorem \ref{theorem:prod-regular} and Theorem \ref{theorem:failureltg} cover Maltsev conditions $\Sigma_G$ for a quite large class of permutation groups $G$, but we are far from a full classification. One of the smallest groups which is not covered by our results is the permutation group $G = \langle (1,2,3)(4,5,6), (1,4)(2,5) \rangle \leq \Sym(6)$. This $G$ is isomorphic to $A_4$ as an abstract group, but it is its action on $\{1,2,\dots,6\}$ that matters here.

In order to understand the local-global property, and the complexity of $\Decide(\Sigma_{G})$ for all $G$, it would be instrumental to classify all conditions $\Sigma_G$ up to interpretability first. This classification problem is interesting independently of deciding $G$-terms in particular, as the class of all $\Sigma_{G}$'s includes several well-known Maltsev conditions.

Although we showed that $n$-ary symmetric terms do not have the local-global property, we still do not know the complexity of deciding them in finite (idempotent) algebras. Thus it is natural to ask:

\begin{question}
How hard is $\Decide(\Sigma_{\Sym(n)})$ and $\Decide^{id}(\Sigma_{\Sym(n)})$ for a fixed $n > 2$?
\end{question}

Note that Theorem \ref{theorem:loop-condition} does not exclude the possibility of another Maltsev condition that is equivalent to the existence of symmetric terms, and which has the local-global property.  It makes sense to ask in general,  whether the local-global property depends on the way a Maltsev condition is presented:

\begin{question} \label{question:equivalence}
Are there (linear strong) Maltsev conditions $\Sigma_1$, $\Sigma_2$, such that $\Sigma_1$ and $\Sigma_2$ are equivalent over all (finite) algebras, but $\Sigma_1$ has the local-global property and $\Sigma_2$ does not?
\end{question}

Note that (unlike in several previous results) the assumption of idempotence did not play a role in proving Theorem \ref{theorem:prod-regular} and Theorem \ref{theorem:failureltg}. Similar observations can be made about the proof of the local-global property for qWNU terms in \cite{kazda-decidingQWNU}. This might be connected to the fact, that our paper and \cite{kazda-decidingQWNU} study Maltsev conditions of height 1, while most previous results were about Maltsev conditions, that are linear but not of height 1. Therefore we conclude with the following questions:

\begin{question}
Let $\Sigma$ be a strong Maltsev condition of height 1. \begin{enumerate}
    \item Does $\Sigma$ have the local-global property if and only if $\Sigma$ has the local-global property for idempotent algebras?
    \item Does $\Decide(\Sigma)$ have the same complexity as $\Decide^{id}(\Sigma)$?
\end{enumerate}
\end{question}

\section{Acknowledgements}
We thank Matt Valeriote for suggesting the question of studying the conditions $\Sigma_G$. Moreover, we would like to thank Attila F\"oldv\'ari for answering some questions about permutation groups and providing the example of the smallest group not covered by our results.

\begin{Backmatter}
\bibliographystyle{apalike}
\bibliography{local-global}
\printaddress
\end{Backmatter}

\end{document}